\documentclass[11pt,openany]{article}
\usepackage{amsmath}
\usepackage{amsthm}
\usepackage{amscd}
\usepackage{amssymb}
\usepackage{tikz}
\newcommand{\g}{\frak g}
\newcommand{\e}{\frak e}
\newcommand{\n}{\frak n}
\newcommand{\pa}{\partial}
\newcommand{\m}{\frak m}

\newcommand{\F}{{\Bbb F}}
\newcommand{\cw}{{\curlywedge}}
\newcommand{\f}{\frac}
\newcommand{\lo}{\longrightarrow}

\newtheorem{theorem}{Theorem}[section]

\newtheorem{corollary}[theorem]{Corollary}

\newtheorem{lemma}[theorem]{Lemma}

\newtheorem{proposition}[theorem]{Proposition}

\theoremstyle{definition}
\newcommand{\ds}{\displaystyle}

\newtheorem{example}[theorem]{Example}

\title{\bf
Some results on the second relative homology\\
 of Leibniz algebras}
\author{Seyedeh Narges Hosseini$^{a}$ , Behrouz Edalatzadeh$^{b}$
, Ali Reza Salemkar$^{a,}$\footnote{
Corresponding author.}\\
{\small $^a$ Department of Mathematics, Faculty of Mathematical
Sciences, Shahid Beheshti University, Tehran, Iran}\\
{\small $^b$ Department of Mathematics, Faculty of Science, Razi
University, Kermanshah, Iran}\\
{\small narges.hosseini90@gmail.com~,~edalatzadeh@gmail.com~,~
salemkar@sbu.ac.ir} }
\date{ }
\begin{document}
\maketitle
\begin{abstract}
In this paper, the structure of the second relative homology and
the relative stem cover of the direct sum of two pairs of Leibniz
algebras are  determined by means of the non-abelian tensor
product of Leibniz algebras. We also characterize all pairs of
finite dimensional nilpotent Leibniz algebras such that
$\dim(\n)=n$, $\dim(\g/\n)=m$ and $\dim(HL_2(\g,\n))=n(n+2m)-3$.
\\[.2cm]
\end{abstract}

\section{Introduction}
Leibniz algebras have been initially appeared in the papers of
Bloh \cite{B,B2,B3} as a non skew-symmetric analogue of Lie
algebras which satisfying a certain condition the so-called
Leibniz rule. They further explored by Loday \cite{Loday1,Loday}
for constructing a new (co)homology theory for Lie algebras, the
so-called Leibniz (co)homology (see also \cite{LP}).  Leibniz
algebras are  naturally applied to several areas of mathematics
and physics such as differential geometry, homological algebra,
algebraic topology, algebraic $K$-theory and  non-commutative
geometry.

Although some theoretic results in Leibniz algebras are
straightforward parallel to Lie algebras, however, there are some
complex difficulties in generalization of some topics, specially
the classification problem of Leibniz algebras, one can see some
distinguished references like \cite{barnes12,Barnes,Batten,Demir}.
Also, see \cite{Demir1,Kh} to compare the vastity of the classes
of low dimensional nilpotent non-Lie Leibniz algebras.

The homology theory of Leibniz algebras was flourished at the
beginning of the development of these algebras in a paper by Loday
and Pirashvili \cite{LP}, and immediately several concepts related
to the theory of abelian extensions tied to low dimensional
homologies, see \cite{Casas2}. From the practical standpoint to
the theory of Baer-invariants, the relative Lie-central extension,
Lie-cover and Lie-multiplier of Leibniz algebras with respect to
the subctegory of Lie algebras were studied in a series of papers,
see for instance \cite{CaKh,CaIn}. The experience of working with
Lie algebras inspired us to employ the second homology of Leibniz
algebras  (which is called the Schur multiplier) and some related
tools to achieve a new points of view to the characterization of
Leibniz algebras, notably for nilpotent ones, see
\cite{Edalatzadeh,edalatzadeh4,salemkar2}. In \cite{edalatzadeh3},
the authors introduced the second relative homology of a pair
$(\g,\n)$ as a term of a spectral sequence, here $\n$ is an ideal
of the Leibniz algebra $\g$.  We also presented the structure of
relative stem cover, a version of the Hopf's formula and a certain
upper bound for $HL_2(\g,\n)$. In addition, we classified all
pairs of finite dimensional nilpotent Leibniz algebras that have
one or two steps distance to this upper bound. Recent further
development has been carried out by Biyogman and Safa
\cite{Biyog}, with emphasis on the universal relative central
extension of a perfect pair of Leibniz algebras $(\g,\n)$ in which
$\n$ admits a complement.

This paper is devoted to derive some fruitful results about the
second relative homology of a pair of Leibniz algebras. In the
first step, as a generalization of the K\"uneth formula for the
second Leibniz homology (see \cite{LodayKuneth}), we determine the
structure of the second relative homology of the pair of Leibniz
algebras $(\g_1\oplus\g_2,\n_1\oplus \n_2)$ in terms of the second
relative homologies of the pairs $(\g_1,\n_1)$ and $(\g_2,\n_2)$.
As a consequence, for an arbitrary pair of finite dimensional
nilpotent Leibniz algebras $(\g,\n)$, we obtain an inequality for
the dimension of $HL_2(\g,\n)$ in terms of the dimension of the
commutator subalgebra and quotient algebras. This inequality
improves the result of \cite{Biyog} in the case that $\n$ admits a
complement in $\g$. Also, we develop the classification obtained
in \cite{edalatzadeh3}, where all pairs of finite dimensional
nilpotent Leibniz algebras in which have three steps distance to
the maximal possible bound for $HL_2(\g,\n)$ are determined.
Finally, we close this paper by describing the structure of the
relative stem cover of the direct sum of two pairs of Leibniz
algebras.
\vspace{.2cm}\\
{\large{\bf Notations}}. Throughout this paper, by a Leibniz
algebra we mean a right Leibniz algebra over some fixed field
$\F$.  We write $\otimes_\F$ for the usual tensor product of
vector spaces over $\F$. If ${\frak a}$ and ${\frak b}$ are
subspaces of a vector space ${\frak k}$ for which ${\frak
k}={\frak a}+{\frak b}$ and ${\frak a}\cap{\frak b}=0$, we will
write ${\frak k}={\frak a}\dot{+}{\frak b}$.
\section{The non-abelian exterior product of Leibniz algebras}
This section is devoted to the study of the properties of the
non-abelian tensor and exterior products of Leibniz algebras. We
begin by recalling these concepts.

Let $\m$ and $\g$ be two Leibniz algebras. By an action of $\g$ on
$\m$ we mean a couple of $\F$-bilinear maps $\g\times\m\to\m,
(x,m)\mapsto\hspace{-.1cm}~^xm$ and $\m\times\g\to\m$,
$(m,x)\mapsto m^x$, satisfying the axioms\vspace{-.13cm}
\begin{alignat*}{2}
&~~^{[x,x']}m =\hspace{-.1cm}~^x(~^{x'}m)
+(\hspace{-.1cm}~^xm)^{x'},
 ~~~~~~~~&&~~^x[m,m']=[\hspace{-.1cm}~^xm,m']-[\hspace{-.1cm}~^xm',m],\\[-.15cm]
&~~m^{[x,x']} = (m^x)^{x'}-(m^{x'})^x, &&~~[m,m']^x =[m^x,m'] +[m,{m'}^x],\\[-.15cm]
&~ ~^x(\hspace{-.1cm}~^{x'}m) = -\hspace{-.1cm}~^x(m^{x'}),&&~~
[m,\hspace{-.1cm}~^xm'] = -[m,{m'}^x],
\end{alignat*}
~\vspace{-.75cm}\\
for all $m,m'\in\m, x,x'\in\g$. Note that if $\g$ is a subalgebra
of some Leibniz algebra $\mathfrak{p}$ and $\m$ is an ideal in
$\mathfrak{p}$, then the Leibniz product in $\mathfrak{p}$ induces
an action of $\g$ on $\m$ given by $\hspace{-.1cm}~^xm=[x,m]$ and
$m^x=[m,x]$. In particular, there is an action of $\mathfrak{p}$
on itself given by the Leibniz product in $\mathfrak{p}$.

A {\it $($Leibniz$)$ crossed module} is a homomorphism of Leibniz
algebras $\mu:\m\to\g$ together with an action of $\g$ on $\m$
such that $\mu(\hspace{-.1cm}~^xm)=[x,\mu(m)]$,
$\mu(m^x)=[\mu(m),x]$ and$~^{\mu(m)}{m'}=[m,m']=m^{\mu(m')}$, for
all $x\in\g$, $m,m'\in\m$. Plainly, if $\m$ is an ideal of $\g$,
then the inclusion map $\m\hookrightarrow\g$ is a crossed module.

Let $\mu:\m\to\g$ and $\lambda:\n\to\g$ be two crossed modules of
Leibniz algebras. Then $\m$ and $\n$ act on each other via the
action of $\g$. The {\it non-abelian tensor product} $\m\ast\n$ is
defined in \cite{Gnedbaye} as the Leibniz algebra generated by the
symbols $m\ast n$ and $n\ast m$ ($m\in\m$, $n\in\n$), subject to
the relations\vspace{-.2cm}
\begin{alignat*}{2}
&(1a)~c(m\ast n)=cm\ast n = m\ast cn, &&(1b)~c(n\ast m) = cn\ast m = n\ast cm,\\[-.2cm]
&(2a)~(m+m')\ast n=m\ast n+m'\ast n,~~~~~~~~~~~&& (2b)~(n+n')\ast m = n\ast m+n'\ast m,\\[-.2cm]
&(2c)~m\ast(n +n')=m\ast n+m\ast n',&&(2d)~ n\ast (m+m') = n\ast m+n\ast m',\\[-.2cm]
&(3a)~m\ast[n,n']=m^n\ast n' -m^{n'}\ast n,&&(3b)~n\ast[m,m'] = n^m\ast m' -n^{m'}\ast m,\\[-.2cm]
&(3c)~[m,m']\ast n=~^mn\ast m' -m\ast n^{m'},&&(3d)~[n,n']\ast m = ~^nm\ast n' -n\ast m^{n'},\\[-.2cm]
&(4a)~m\ast~^{m'}n=-m\ast n^{m'},&& (4b)~n\ast~^{n'}m=-n\ast m^{n'},\\[-.2cm]
&(5a)~m^n\ast~^{m'}n'= [m\ast n,m'\ast n'] = ~^mn\ast m'^{n'},~&&
(5b)~^nm\ast n'^{m'}=[n\ast m,n'\ast m']=n^m\ast~^{n'}m',\\[-.2cm]
&(5c)~m^n\ast n'^{m'}=[m\ast n,n'\ast
m']=^mn\ast\hspace{-.cm}~^{n'}m',&& (5d)\hspace{-.cm}~^nm\ast
~^{m'}n'=[n\ast m,m'\ast n']=n^m\ast m'^{n'},
\end{alignat*}
~\vspace{-.75cm}\\for all $c\in\F$, $m,m'\in\m$, $n,n'\in \n$.
Note that the identity map $id_\g:\g\lo\g$ is a crossed module, so
we can form the tensor products $\g\ast\m$, $\g\ast\n$ and
$\g\ast\g$.

Let $\m\square\n$ be the vector subspace of $\m\ast\n$ spanned by
the elements $m\ast n'-n\ast m'$ with $\mu(m)=\lambda(n)$ and
$\mu(m')=\lambda(n')$. One easily gets that $\m\square\n$ lies in
the center of $\m\ast\n$. The {\it non-abelian exterior product}
$\m\curlywedge\n$ is defined to be the quotient
$(\m\ast\n)/(\m\square\n)$. We write $m\curlywedge n$ and
$n\curlywedge m$ to denote the images in $\m\curlywedge\n$ of the
generators $m\ast n$ and $n\ast m$, respectively.

The following results give some information and properties of the
above notions that will be needed.
\begin{lemma}[{\rm\cite{Donadze}}]\label{lem2.1}
Let $\g$ be a Leibniz algebra with ideals $\n$ and $\mathfrak{k}$.

$(i)$ If $\n$ and $\mathfrak{k}$ act trivially on each other, then
$\n\ast\mathfrak{k}\cong
\n^{ab}\ast\mathfrak{k}^{ab}\cong(\n^{ab}\otimes_{\F}\mathfrak{k}^{ab})
\oplus(\mathfrak{k}^{ab}\otimes_{\F}\n^{ab})$, where
$\n^{ab}=\n/\n^2$ and
$\mathfrak{k}^{ab}=\mathfrak{k}/\mathfrak{k}^2$.

$(ii)$ If $\mathfrak{k}\subseteq\n$, then there is the following
natural exact sequence of Leibniz algebras\vspace{-.2cm}
\[\g\curlywedge\mathfrak{k}\stackrel{\alpha}\lo\g\curlywedge\n
\twoheadrightarrow\f{\g}{\mathfrak{k}}\curlywedge\f{\n}{\mathfrak{k}}.\vspace{-.2cm}\]
Furthermore, if $\g=\n$ and $\mathfrak{k}$ has an ideal complement
in $\g$, then $\alpha$ is injective.
\end{lemma}
\begin{proposition}[{\rm\cite{salemkar}}]
Let $\g$ be any Lie algebra. Then there is an isomorphism of Lie
algebras $\beta:\g\curlywedge\g\stackrel{\cong}\lo\g\otimes\g$
defined on generators by $\beta(x_1\curlywedge x_2)=x_1\otimes
x_2$, where $\otimes$ denotes the non-abelian tensor product of
Lie algebras, introduced in {\rm \cite{Ellis}}.
\end{proposition}
\begin{proposition}\label{prop2.3}
Let $(\mathfrak{g}_1,\mathfrak{n}_1)$ and
$(\mathfrak{g}_2,\mathfrak{n}_2)$ be arbitrary pairs of Leibniz
algebras. Then\vspace{-.1cm}
\[(\mathfrak{g}_1\oplus\mathfrak{g}_2)\curlywedge(\mathfrak{n}_1\oplus\mathfrak{n}_2)
\cong(\mathfrak{g}_1\curlywedge\mathfrak{n}_1)\oplus
(\mathfrak{g}_2\curlywedge\mathfrak{n}_2)\oplus\ds
\f{(\overline{\mathfrak{n}_1}\ast\overline{\mathfrak{g}_2})\oplus
(\overline{\mathfrak{g}_1}\ast\overline{\mathfrak{n}_2})}{\mathfrak{a}}.\vspace{-.2cm}\]
where
$\overline{\mathfrak{n}_i}=\mathfrak{n}_i/[\mathfrak{g}_i,\mathfrak{n}_i]$
and $\overline{\mathfrak{g}_i}=\mathfrak{g}_i/[\mathfrak{g}_i
,\mathfrak{g}_i]$ for $i=1,2$, and $\mathfrak{a}$ is the
subalgebra generated by the elements $(\bar n_1\ast\bar n_2,-\bar
n_1\ast\bar n_2)$ and $(\bar n_2\ast\bar n_1,-\bar n_2\ast\bar
n_1)$ in which $n_1\in\mathfrak{n}_1$, $n_2\in\mathfrak{n}_2$
$($here the bar $\bar.$ denotes the equivalence class in each
case$)$.
\end{proposition}
\begin{proof}
Define the map
$\varphi:(\mathfrak{g}_1\oplus\mathfrak{g}_2)\curlywedge(\mathfrak{n}_1\oplus\mathfrak{n}_2)
\lo(\mathfrak{g}_1\curlywedge\mathfrak{n}_1)\oplus
(\mathfrak{g}_2\curlywedge\mathfrak{n}_2)\oplus\ds
\f{(\overline{\mathfrak{n}_1}\ast\overline{\mathfrak{g}_2})\oplus
(\overline{\mathfrak{g}_1}\ast\overline{\mathfrak{n}_2})}{\mathfrak{a}}$
on generators as follows:\vspace{-.26cm}
\[\hspace{.3cm}\varphi((x_1,x_2)\curlywedge(n_1,n_2))=(x_1\curlywedge n_1,x_2\curlywedge n_2,
(\bar n_1\ast\bar x_2,\bar n_2\ast\bar
x_1)+\mathfrak{a}),\vspace{-.3cm}\]
\[\hspace{.3cm}\varphi((n_1,n_2)\curlywedge(x_1,x_2))=(n_1\curlywedge x_1,n_2\curlywedge x_2,
(\bar x_2\ast\bar n_1,\bar x_1\ast\bar
n_2)+\mathfrak{a}),\vspace{-.15cm}\]for all $x_i\in\g_i$,
$n_i\in\n_i$, $i=1,2$. It is not difficult to verify that
$\varphi$ is well-defined and preserves the defining relations of
the exterior product. For instance, we indicate that\vspace{-.2cm}
\begin{equation}
\varphi([(x_1,x_2),(x_1',x_2')]\curlywedge(n_1,n_2))=
\varphi(\hspace{-.1cm}~^{(x_1,x_2)}(n_1,n_2)
\curlywedge(x_1',x_2'))-\varphi((x_1,x_2)\curlywedge(n_1,n_2)^{(x_1',x_2')}).\vspace{-.15cm}
\end{equation}
We have\vspace{-.4cm}
\begin{alignat*}{1}
\varphi([(x_1,x_2),(x_1',x_2')]\curlywedge(n_1,n_2))&=
\varphi(([x_1,x_1'],[x_2,x_2'])\curlywedge(n_1,n_2))\\
&=([x_1,x_1']\curlywedge n_1,[x_2,x_2']\curlywedge n_2,(\bar
n_1\ast\overline{[x_2,x_2']},\bar
n_2\ast\overline{[x_1,x_1']})+\mathfrak{a})\\
&=([x_1,x_1']\curlywedge n_1,[x_2,x_2']\curlywedge
n_2,\mathfrak{a}).
\end{alignat*}
~\vspace{-1cm}\\ On the other hand,\vspace{-.2cm}
\begin{alignat*}{1}
\varphi(&\hspace{-.1cm}~^{(x_1,x_2)}(n_1,n_2)
\curlywedge(x_1',x_2'))-\varphi((x_1,x_2)\curlywedge(n_1,n_2)^{(x_1',x_2')})\\
=&(\hspace{-.1cm}~^{x_1}n_1\curlywedge x_1'-x_1\curlywedge
n_1^{x_1'},\hspace{-.1cm}~^{x_2}n_2\curlywedge x_2'-x_2\curlywedge
n_2^{x_2'},
(\overline{x_2'}\ast\overline{\hspace{-.1cm}~^{x_1}n_1}-\overline{n_1^{x_1'}}\ast
\overline{x_2},\overline{x_1'}\ast\overline{\hspace{-.1cm}~^{x_2}n_2}
-\overline{n_2^{x_2'}}\ast\overline{x_1})+\mathfrak{a})\\
=&([x_1,x_1']\curlywedge n_1,[x_2,x_2']\curlywedge
n_2,\mathfrak{a}),
\end{alignat*}
~\vspace{-.8cm}\\because $\hspace{-.1cm}~^{x_1}n_x,
n_1^{x_1'}\in[\mathfrak{g}_1,\mathfrak{n}_1]$ and
$\hspace{-.1cm}~^{x_2}n_2,
n_2^{x_2'}\in[\mathfrak{g}_2,\mathfrak{n}_2]$. Then the equality
$(1)$ holds. We now prove that $\varphi$ is an isomorphism by
giving an inverse for it. To do this, let us first define the
maps\vspace{-.15cm}
\[\eta_1:\overline{\mathfrak{n}_1}\ast\overline{\mathfrak
    {g}_2}\lo
(\mathfrak{g_1}\oplus\mathfrak{g}_2)\curlywedge(\mathfrak{n}_1\oplus\mathfrak{n}_2)
~~~~~~{\rm and}~~~~~~
\eta_2:\overline{\mathfrak{g}_1}\ast\overline{\mathfrak{n}_2}\lo
(\mathfrak{g}_1\oplus\mathfrak{g}_2)\curlywedge
(\mathfrak{n}_1\oplus\mathfrak{n}_2)\vspace{-.15cm}\] by
$\eta_1(\bar n_1\ast\bar x_2)=(0,x_2)\curlywedge(n_1,0)$,
$\eta_1(\bar x_2\ast\bar n_1)=(n_1,0)\curlywedge(0,x_2)$ and
$\eta_2(\bar x_1\ast\bar n_2)=(0,n_2)\curlywedge(x_1,0)$,
$\eta_2(\bar n_2\ast\bar x_1)=(l_1,0)\curlywedge(0,n_2)$. For
$n_1,n_1'\in\mathfrak{n}_1$, $x_2,x_2'\in\mathfrak{g}_2$,
$y\in[\mathfrak{g}_1,\mathfrak{n}_1]$,
$z\in[\mathfrak{g}_2,\mathfrak{g}_2]$, if $n_1=n_1'+y$ and
$x_2=x_2'+z$, then we have\vspace{-.2cm}
\[\eta_1(\bar n_1\ast\bar x_2)=(0,x_2'+z)\curlywedge(n_1'+y,0)
=(0,x_2')\curlywedge(n_1',0)+(0,x_2')\curlywedge(y,0)
+(0,z)\curlywedge(n_1'+y,0).\vspace{-.14cm}\]But, assuming
$y=[a,b]$ and $z=[c,d]$ for some $a\in\mathfrak{g}_1$,
$b\in\mathfrak{n}_1$, $c,d\in\mathfrak{g}_2$, we have\\
~\hspace{.4cm}$(0,x_2')\curlywedge(y,0)=(0,x_2')\curlywedge([a,b],0)=(0,x_2')^{(a,0)}
\curlywedge(b,0)-(0,x_2')^{(b,0)}\curlywedge(a,0)=0$,\\
~\hspace{.4cm}$(0,z)\curlywedge(n_1'+y,0)=(0,[c,d])\curlywedge(n_1'+y,0)=
\hspace{-.1cm}~^{(0,c)}(n_1'+y,0)\curlywedge(0,d)-(0,c)\curlywedge
(n_1'+y,0)^{(0,d)}=0$,\\
since $\mathfrak{g}_1\oplus 0$ and $0\oplus\mathfrak{g}_2$ act
trivially on each other. Consequently, $\eta_1(\bar n_1\ast\bar
x_2)=\eta_1(\bar n_1'\ast\bar x_2')$ and, by an analogous
argument, $\eta_1(\bar x_2\ast\bar n_1)=\eta_1(\bar x_2'\ast\bar
n_1')$. Therefore, $\eta_1$ and similarly, $\eta_2$ are
well-defined. It is readily checked that $\eta_1$ and $\eta_2$ are
Leibniz homomorphisms whose images are the abelian subalgebras of
$(\mathfrak{g}_1\oplus\mathfrak{g}_2)\curlywedge(\mathfrak{n}_1\oplus\mathfrak{n}_2)$.
Hence, we can obtain a Leibniz homomorphism
$\eta:(\overline{\mathfrak{n}_1}\ast\overline{\mathfrak{g}_2})\oplus
(\overline{\mathfrak{g}_1}\ast\overline{\mathfrak{n}_2})\lo
(\mathfrak{g}_1\oplus\mathfrak{g}_2)\curlywedge(\mathfrak{n}_1\oplus\mathfrak{n}_2)$
defined by $\eta(v,w)=\eta_1(v)+\eta_2(w)$. As $\mathfrak{a}$ is
annihilated by $\eta$, this gives rise to a Leibniz
homomorphism\vspace{-.2cm}
\[\bar\eta:\f{(\overline{\mathfrak{n}_1}\ast\overline{\mathfrak{g}_2})\oplus
(\overline{\mathfrak{g}_1}\ast\overline{\mathfrak{n}_2})}{\mathfrak{a}}
\lo(\mathfrak{g}_1\oplus\mathfrak{g}_2)\curlywedge(\mathfrak{n}_1
\oplus\mathfrak{n}_2).\vspace{-.15cm}\] Now, if we consider the
canonical homomorphisms\vspace{-.2cm}
\begin{alignat*}{2}
\psi_1:\mathfrak{g}_1\curlywedge\mathfrak{n}_1&\lo
(\mathfrak{g}_1\oplus\mathfrak{g}_2)\curlywedge(\mathfrak{n}_1
\oplus\mathfrak{n}_2),~~~x_1\curlywedge
n_1\longmapsto(x_1,0)\curlywedge(n_1,0),~~~n_1\curlywedge
x_1\longmapsto(n_1,0)\curlywedge(x_1,0),\\
\psi_2:\mathfrak{g}_2\curlywedge\mathfrak{n}_2&\lo
(\mathfrak{g}_1\oplus\mathfrak{g}_2)\curlywedge(\mathfrak{n}_1
\oplus\mathfrak{n}_2),~~~x_2\curlywedge
n_2\longmapsto(0,x_2)\curlywedge(0,n_2),~~~n_2\curlywedge
x_2\longmapsto(0,n_2)\curlywedge(0,x_2),\vspace{-.2cm}
\end{alignat*}
then the homomorphism $\psi=\langle\psi_1,\psi_2,\bar\eta\rangle$
in the coproduct of vector spaces\vspace{-.15cm}
\[\psi:(\mathfrak{g}_1\curlywedge\mathfrak{n}_1)\oplus
(\mathfrak{g}_2\curlywedge\mathfrak{n}_2)\oplus\ds
\f{(\overline{\mathfrak{n}_1}\ast\overline{\mathfrak{g}_2})\oplus
(\overline{\mathfrak{g}_1}\ast\overline{\mathfrak{n}_2})}{\mathfrak{a}}
\lo(\mathfrak{g}_1\oplus\mathfrak{g}_2)\curlywedge(\mathfrak{n}_1
\oplus\mathfrak{n}_2),\vspace{-.15cm}\] is evidently an inverse
for $\varphi$. This completes the proof.
\end{proof}
Proposition 2.3 provides a description for the exterior product of
the direct sum of Leibniz algebras.
\begin{corollary}\label{coro2.6}
For any two Leibniz algebras $\g_1$ and $\g_2$, there is a Leibniz
isomorphism\vspace{-.2cm}
\[(\g_1\oplus\g_2)\curlywedge(\g_1\oplus\g_2)
\cong(\g_1\curlywedge\g_1)\oplus
(\g_2\curlywedge\g_2)\oplus(\g_1^{ab}\ast\g_2^{ab}).\vspace{-.2cm}\]
\end{corollary}
\begin{proof}
It suffices to note that, in this case, $\mathfrak{a}=\langle
(a,-a)~|~ a\in\g_1^{ab}\ast\g_2^{ab}\rangle$.
\end{proof}
Combining the above corollary with Lemma 2.1$(i)$ and Proposition
2.2, we obtain the following important result, which was already
proved in Ellis (1991) using another technique.
\begin{corollary}\label{coro2.7}
For any two Lie algebras $\g_1$ and $\g_2$, there is a Lie
isomorphism \vspace{-.2cm}
\[(\g_1\oplus\g_2)\otimes(\g_1\oplus\g_2)
\cong(\g_1\otimes\g_1)\oplus
(\g_2\otimes\g_2)\oplus(\g_1^{ab}\otimes_{\F}\g_2^{ab})
\oplus(\g_2^{ab}\otimes_{\F}\g_1^{ab}).\vspace{-.2cm}\]
\end{corollary}
Let $(\g,\n)$ be a pair of Lie algebras. In general, the Lie
algebras $\g\curlywedge\n$ and $\g\otimes\n$ are not isomorphic.
For example, if $(\g,\n)$ is a pair of abelian Lie algebras with
$\dim(\n)=n$ and $\dim(\g/\n)=m$, then
$\dim(\g\curlywedge\n)=n(n+2m)$, while $\dim(\g\otimes\n)=n(n+m)$.
However, the following proposition provides, under some condition,
a precise relationship between these Lie algebras.
\begin{proposition}\label{prop2.8}
Let $(\g,\n)$ be a pair of Lie algebras. If $\n$ has an ideal
complement in $\g$, then\vspace{-.2cm}
\[\g\curlywedge\n\cong(\g\otimes\n)\oplus(\f{\n}{[\g,\n]}
\otimes_\F\f{\g}{\g^2+\n}).\vspace{-.2cm}\]In particular, if the
pair $(\g,\n)$ is perfect $($that is, $[\g,\n]=\n)$, then
$\g\curlywedge\n\cong(\g\otimes\n)$.
\end{proposition}
\begin{proof}
By hypothesis, there is a Lie decomposition
$\g=\n\oplus\mathfrak{k}$ and also, $[\g,\n]=\n^2$. Invoking
Corollary \ref{coro2.7} and \cite[Proposition 5 and Lemma
7]{Ellis}, we have
$\g\otimes\g\cong(\g\otimes\n)\oplus(\mathfrak{k}\otimes\mathfrak{k})
\oplus(\n^{ab}\otimes_{\F}\mathfrak{k}^{ab})$.
Consider the following commutative diagram of Lie algebras with
exact rows:\vspace{.2cm}
\begin{center}
    \tikzset{node distance=1.cm, auto}
    \begin{tikzpicture}
    \hspace{-.6cm}\node(M1) at(1,3) {$\g\curlywedge\n$};
    \node(M2) at(4,3) {$\g\curlywedge\g$};
    \node(M3) at(7,3) {$\mathfrak{k}\curlywedge\mathfrak{k}$};
    \node(MM1) at(1,1.5) {$(\g\otimes\n)\oplus(\n^{ab}\otimes_{\F}\mathfrak{k}^{ab})$};
    \node (MM2) at(4,1.5) {$\g\otimes\g$};
    \node (MM3) at(7,1.5) {$\mathfrak{k}\otimes\mathfrak{k}$,};
    \draw[>->] (M1) to node{}  (M2);
    \draw[->>] (M2) to node{}  (M3);
    \draw[>->] (MM1) to node{}  (MM2);
    \draw[->>] (MM2) to node{}  (MM3);
    \draw[->] (M1) to node{{\small$\beta_1|$}} (MM1);
    \draw[->] (M2) to node{{\small$\beta_1$}} (MM2);
    \draw[->] (M3) to node{{\small$\beta_2$}} (MM3);
    \end{tikzpicture}
\end{center}
~\vspace{-1cm}\\ where $\beta_1$ and $\beta_2$ are Lie
isomorphisms given in Proposition 2.2, and $\beta_1|$ is the
restriction of $\beta_1$. Applying the short five lemma,
$\beta_1|$ is a Lie isomorphism, as we wished to prove.
\end{proof}
\section{The relative homology of Leibniz algebras}
Let $\g$ be a Leibniz algebra over an arbitrary field $\F$. The
Leibniz homology of $\g$ with trivial coefficients, denoted by
$HL_\ast(\g)$, is the homology of the chain complex
$(CL_n(\g)=\g^{\otimes n},\pa_n)$, $n\geq0$, such that the
boundary map $\pa_n:CL_n(\g)\to CL_{n-1}(\g)$ is a linear map of
vector spaces given by\vspace{-.1cm}
$$\pa_n(x_1\otimes\cdots\otimes x_n)=\sum_{1\leq i<j\leq n}(-1)^{j}(x_1
\otimes\cdots\otimes[x_i,x_j]\otimes\cdots\otimes\widehat{x_j}\otimes\cdots\otimes
x_n)~~~,~~~n\geq1,\vspace{-.1cm}$$where the notation
$\widehat{x_j}$ means that the variable $x_j$ is omitted. It is
easily seen that $HL_(\g)\cong\g^{ab}$ and by \cite[Theorem
3.16]{Donadze}, $HL_2(\g)\cong\ker(\g\curlywedge\g\lo\g)$. Let
$\n$ be an ideal of $\g$, then the natural epimorphism
$\pi:\g\to\g/\n$ induces a morphism of chain complexes
$\pi_*:(CL_*(\g),\pa_*)\to (CL_*(\g/\n),\bar\pa_*)$. Let
$(M_*,\delta_*)=M(\pi_*)$ be the mapping cone of $\pi_*$, that is,
$M_n=CL_{n-1}(\g)\oplus CL_n(\g/\n)$ and
$\delta_n(a,b)=(-\pa_{n-1}(a),\bar\pa_n(b)+\pi_n(a))$. Setting
$(K^+_*,\pa_*^+)$ to be the complex $(CL_*(\g),\pa_*)$ with the
dimensions all raised by one and the sign of the boundary changed,
that is $(K^+)_n=CL_{n-1}(\g)$, then there exists  a short exact
sequence of complexes $0\to (CL_*(\g/\n),\bar\pa_*)\to
(M_*,\delta_*)\to(K^+_*,\pa_*^+)\to 0$. This sequence induces a
long exact sequence of homology groups\vspace{-.1cm}
\begin{equation}\label{seq}
\cdots\to HL_n(\g/\n)\to H_n(M_*,\delta_*)\to HL_{n-1}(\g)\to
HL_{n-1}(\g/\n)\to\cdots,\vspace{-.1cm}
\end{equation}
see \cite[Page 47]{Maclane} for more details. The abelian Leibniz
algebra $H_{n+1}(M_*,\delta_*)$ is called the $n$-th {\it relative
homology} of the pair $(\g,\n)$ and denoted by $HL_n(\g,\n)$. It
was shown in \cite[Propositions 2,4]{Casas} that $HL_1(\g,\n)\cong
\n/[\g,\n]$ and if $\n$ is a central ideal of $\g$, then
$HL_2(\g,\n)\cong {\rm Coker}(\tau)$, where $\tau:\n\otimes \n\to
(\g/\g^2\otimes \n)\oplus (\n\otimes\g/\g^2)$ is defined by
$\tau(n_1\otimes n_2)=(\bar n_1\otimes n_2,-n_1\otimes \bar n_2)$
for all $n_1,n_2\in\n$. The following proposition gives some
different descriptions of $HL_2(\g,\n)$.
\begin{proposition}\label{prop2.2}
Let $(\g,\n)$ be a pair of Leibniz algebras. Then

$(i)$ $HL_2(\g,\n)$ is isomorphic to the kernel of the commutator
map $\g\curlywedge\n\stackrel{[~,~]}\lo\g$.

$(ii)$ If $\n$ is central, then
$HL_2(\g,\n)\cong(\g^{ab}\ast\n)/\langle\bar n\ast n,n\ast\bar
n~|~n\in\n\rangle$, where $\bar n$ denotes the image of $n$ in
$\g^{ab}$. If in addition $\n\subseteq\g^2$, then
$HL_2(\g,\n)\cong\g^{ab}\ast\n$.

$(iii)$ If $\n$ admits a complement in $\g$, then
$HL_2(\g,\n)\cong\ker(HL_2(\g)\twoheadrightarrow HL_2(\g/\n))$ and
$HL_2(\g)\cong HL_2(\g,\n)\oplus HL_2(\g/\n)$.

$(iv)$ If
$\mathfrak{r}\rightarrowtail\mathfrak{f}\stackrel{\pi}\twoheadrightarrow
\mathfrak{g}$ is a free presentation of $\mathfrak{g}$, where
$\mathfrak{f}$ is a free Leibniz algebra on the set
$\mathfrak{g}$, and $\mathfrak{y}$ is an ideal of $\mathfrak{f}$
generated by the set $\mathfrak{n}$, then
$HL_2(\mathfrak{g},\mathfrak{n})\cong\ds\f{\mathfrak{r}\cap[\mathfrak{y},\mathfrak{f}]}
{[\mathfrak{r},\mathfrak{y}]+[\mathfrak{f},\mathfrak{r}\cap\mathfrak{y}]}$.
\end{proposition}
\begin{proof}
Parts $(i)$ and $(iv)$ are proved in \cite[Proposition
4.3]{Donadze} and \cite[Proposition 2.1]{edalatzadeh3},
respectively. Part $(ii)$ is a consequence of part $(i)$ and Lemma
2.1$(i)$. Part $(iii)$ is obtained from the sequence $(2)$.
\end{proof}
By Propositions \ref{prop2.8} and \ref{prop2.2}$(i)$, we have the
following which generalizes \cite[Proposition 6.2]{Donadze}
slightly.
\begin{corollary}\label{coro2.9}
Under the assumptions of Proposition $\ref{prop2.8}$, there is an
isomorphism of vector spaces\vspace{-.2cm}
\[HL_2(\g,\n)\cong\ker(\g\otimes\n\lo\g)\oplus(\f{\n}{[\g,\n]}
\otimes_\F\f{\g}{\g^2+\n}).\vspace{-.2cm}\]In particular, if the
pair $(\g,\n)$ is perfect, then $HL_2(\g,\n)\cong H_2(\g,\n)$,
where $H_2(\g,\n)$ denotes the second relative Chevalley-Eilenberg
homology of the pair $(\g,\n)$.
\end{corollary}
Let $\g_1,\g_2$ be two Leibniz algebras. The K\"unneth-style
formula for the homology of direct sum of Leibniz algebras
$\g_1,\g_2$, was extended in \cite{LodayKuneth} by Loday.  He
proved that there is a canonical isomorphism of graded modules
$HL_*(\g_1\oplus\g_2)\cong HL_*(\g_1)\star HL_*(\g_2)$ (here the
symbol ``$\star$" means the non-commutative tensor product of
graded modules). In the special case, for the second degree,
theorem states that $HL_2(\g_1\oplus\g_2)\cong HL_2(\g_1)\oplus
HL_2(\g_2)\oplus
(\g_1^{ab}\otimes_{\F}\g_2^{ab})\oplus(\g_2^{ab}\otimes_{\F}
\g_1^{ab})$. Note that one can use Lemma 2.1$(i)$ to re-arrange
the formula as\vspace{-.2cm}
\begin{equation}\label{eq1}
HL_2(\g_1\oplus\g_2)\cong HL_2(\g_1)\oplus
HL_2(\g_2)\oplus(\g_1^{ab}\ast\g_2^{ab}).\vspace{-.2cm}
\end{equation}
As an immediate consequence of Propositions 2.3 and
\ref{prop2.2}$(i)$, we can obtain the following generalization of
the formula $(3)$.
\begin{corollary}\label{coro3.3}
$($The generaliztion of K\"unneth-Loday formula$)$ Let
$(\mathfrak{g}_1,\mathfrak{n}_1)$ and
$(\mathfrak{g}_2,\mathfrak{n}_2)$ be arbitrary pairs of Leibniz
algebras. Then\vspace{-.1cm}
\[HL_2(\mathfrak{g}_1\oplus\mathfrak{g}_2,\mathfrak{n}_1\oplus\mathfrak{n}_2)
\cong HL_2(\mathfrak{g}_1,\mathfrak{n}_1)\oplus
HL_2(\mathfrak{g}_2,\mathfrak{n}_2)\oplus\ds
\f{(\overline{\mathfrak{n}_1}\ast\overline{\mathfrak{g}_2})\oplus
(\overline{\mathfrak{g}_1}\ast\overline{\mathfrak{n}_2})}{\mathfrak{a}}.\vspace{-.1cm}\]
\end{corollary}
To state another interesting consequence of Proposition
\ref{prop2.2}, we first need the following.

A finite dimensional nilpotent Leibniz algebra $\g$ is called {\it
extra special} if $\dim(Z(\g))=\dim(\g^2)=1$. We denote an extra
special Leibniz algebra by $\e$ and an abelian Leibniz algebra of
dimension $q$ by $\mathfrak{a}(q)$. The following proposition
gives the dimension of the second homology of finite dimensional
nilpotent Leibniz algebras with derived subalgebra of dimension
$1$.
\begin{proposition}[\cite{edalatzadeh3}]\label{prop2.3}
$(i)$ $\dim(HL_2(\e))=(\dim(\e)-1)^2-1+t$, for some integer
$t\leq2\dim(\e)$. In particular, if the ground field is
algebraically closed of characteristic different from two, then
$t=1$ if $\e=J_1$ or $\e=J_2$, where $J_1=\langle
x,y~|~[x,x]=y\rangle$ and $J_2=\langle x,y,z~|~[x,y]=z\rangle$;
and $t=2$ if $\e=H_1$, where $H_1=\langle
x,y,z~|~[x,y]=z=-[y,x]\rangle$; and $t=0$ otherwise.

$(ii)$ If $\g$ is a finite dimensional nilpotent Leibniz algebra
with $\dim(\g^2)=1$, then $\g=\e\oplus\mathfrak{a}(q)$ for some
$q\geq0$, and so $\dim(HL_2(\g))=(\dim(\g)-1)^2-1+t$.
\end{proposition}
In the following corollary, we get the dimension of the second
relative homology of some known pairs of Leibniz algebras.
\begin{corollary}\label{coro2.4}
$(i)$ If $\n$ is an $n$-dimensional ideal of $\mathfrak{a}(q)$,
then $\dim(HL_2(\mathfrak{a}(q),\n))=n(2q-n)$.

$(ii)$ For any finite dimensional nilpotent Leibniz algebra $\g$
with $\dim(\g^2)=1$, $\dim(HL_2(\g,\g^2))=2\dim(\g^{ab})$.

$(iii)$ For any non-zero ideal $\n$ of $J_2$ or $H_1$, we have \vspace{.25cm}\\
\vspace{1cm}~\hspace{.7cm}$\dim(HL_2(J_2,\n))=
\begin{cases}
3 & {\rm if}~
\dim(\n)=2\\
4 & {\rm otherwise}
\end{cases}$\vspace{-.8cm}
\vspace{1cm}~\hspace{.55cm},\hspace{.55cm}$\dim(HL_2(H_1,\n))=
\begin{cases}
5 & {\rm if}~
\n=H_1\\
4 & {\rm otherwise}.
\end{cases}$\vspace{-.8cm}
\end{corollary}
\begin{proof}
$(i)$ It follows from Proposition \ref{prop2.2}$(iii)$ and the
fact that any ideal of an abelian Leibniz algebra has a
complement.

$(ii)$ It is a straightforward consequence of Proposition
\ref{prop2.2}$(ii)$

$(iii)$ We only prove the case $H_1$, the proof of the other case
is similar. According to Proposition \ref{prop2.3}$(i)$,
$\dim(HL_2(H_1,H_1))=\dim(HL_2(H_1))=5$ and by part $(ii)$,
$\dim(HL_2(H_1,H_1^2))=4$. Now, suppose that $\n$ is a
two-dimensional ideal of $H_1$. Then $\n$ has a one-dimensional
complement, say $\mathfrak{k}$, in $H_1$ and hence, owing to
Proposition \ref{prop2.2}$(iii)$,
$\dim(HL_2(H_1,\n))=\dim(HL_2(H_1))-\dim(HL_2(\mathfrak{k}))=4$.
\end{proof}
It is proved in \cite{edalatzadeh3} that for any pair $(\g,\n)$ of
finite dimensional nilpotent Leibniz algebras, if $\n$ has a
complement in $\g$, then\vspace{-.1cm}
\[\dim(HL_2(\mathfrak{g},\mathfrak{n}))\leq\dim(HL_2(\f{\mathfrak{g}}
{[\mathfrak{g},\mathfrak{n}]},
\f{\mathfrak{n}}{[\mathfrak{g},\mathfrak{n}]}))+2\dim([\mathfrak{g},\mathfrak{n}])
(d(\f{\mathfrak{g}}{Z(\g,\n)})-1)+
\dim([\mathfrak{g},\mathfrak{n}]),\vspace{-.1cm}\] where
$Z(\g,\n)=Z(\g)\cap\n$ and $d(\g)$ is the cardinal of a minimal
generating set of $\g$. Note that in this case,
$\mathfrak{g}^2\cap\mathfrak{n}=[\mathfrak{g},\mathfrak{n}]$.
Using Corollary 3.3, we prove a similar result for any pair of
finite dimensional nilpotent Leibniz algebras.
\begin{theorem}\label{theo3.5}
Let $\mathfrak{g}$ be a finite dimensional nilpotent Leibniz
algebra with an ideal $\mathfrak{n}$. Then\vspace{-.1cm}
\[\dim(HL_2(\mathfrak{g},\mathfrak{n}))\leq\dim(HL_2(\f{\mathfrak{g}}
{\mathfrak{g}^2\cap\mathfrak{n}},
\f{\mathfrak{n}}{\mathfrak{g}^2\cap\mathfrak{n}}))+2\dim(\mathfrak{g}^2\cap\mathfrak{n})
d(\f{\mathfrak{g}}{Z(\mathfrak{g},\mathfrak{n})}).\vspace{-.1cm}\]
\end{theorem}
Our proof of the above theorem requires the following proposition.
\begin{proposition}\label{prop3.6}
Let $\g$ be a Leibniz algebra with ideals $\n$ and $\mathfrak{k}$
such that $\mathfrak{k}\subseteq Z(\g,\n)$. Then there is an exact
sequence of Leibniz algebras\vspace{-.2cm}
\[\g\curlywedge\mathfrak{k}\lo HL_2(\g,\n)\lo
HL_2(\f{\g}{\mathfrak{k}},\f{\n}{\mathfrak{k}})
\twoheadrightarrow[\g,\n]\cap\mathfrak{k}.\vspace{-.2cm}\]
\end{proposition}
\begin{proof}
By Lemma 2.1$(ii)$, we have the following diagram with exact
rows\vspace{.2cm}
\begin{center}
    \tikzset{node distance=2cm, auto}
    \begin{tikzpicture}
    \node(M1) at(1,3) {$\g\curlywedge\mathfrak{k}$};
    \node(M2) at(4,3) {$\g\curlywedge\n$};
    \node(M3) at(7,3) {$\ds\f{\g}{\mathfrak{k}}\curlywedge\f{\n}{\mathfrak{k}}$};
    \node(MM1) at(1,1.3) {$[\g,\n]\cap\mathfrak{k}$};
    \node (MM2) at(4,1.3) {$[\g,\n]$};
    \node(MM3) at(7,1.3) {$[\ds\f{\g}{\mathfrak{k}},\f{\n}{\mathfrak{k}}]$,};
    \draw[->] (M1) to node{}  (M2);
    \draw[->>] (M2) to node{}  (M3);
    \draw[>->] (MM1) to node{}  (MM2);
    \draw[->>] (MM2) to node{}  (MM3);
    \draw[->] (M1) to node{{\small$[~,~]_1$}} (MM1);
    \draw[->] (M2) to node{{\small$[~,~]_2$}} (MM2);
    \draw[->] (M3) to node{{\small$[~,~]_3$}} (MM3);
    \end{tikzpicture}
\end{center}
~\vspace{-1cm}\\where the vertical arrows are the commutator maps.
In this diagram, the right-hand-side square is always commutative.
Since $\mathfrak{k}$ is central, the commutator map $[~,~]_1$ is
equal to the zero map and so the left-hand-side square is also
commutative. The Snake Lemma yields the following exact
sequence\vspace{-.2cm}
\[\ker([~,~]_1)\lo\ker([~,~]_2)\lo\ker([~,~]_3)
\twoheadrightarrow{\rm coker}([~,~]_1)\vspace{-.2cm}.\]The result
now follows from Proposition \ref{prop2.2}$(i)$.
\end{proof}
We get the following corollary from the  previous proposition.
\begin{corollary}\label{coro3.7}
With the assumptions of Proposition $\ref{prop3.6}$, if $\g$ is
finite dimensional, then\vspace{-.15cm}
\[\dim(HL_2(\f{\g}{\mathfrak{k}},\f{\n}{\mathfrak{k}}))
\leq\dim(HL_2(\g,\n))+\dim([\g,\n]\cap\mathfrak{k}).\vspace{-.15cm}\]
In particular, if the ideal $\n$ is central, then
$\dim(HL_2(\g/\mathfrak{k},\n/\mathfrak{k}))\leq\dim(HL_2(\g,\n))$.
\end{corollary}
\begin{proof}[Proof of Theorem $\ref{theo3.5}$]
We proceed by induction on $\dim(\g)$. The result certainly holds
when $\g$ is abelian. Suppose that $\dim(\g)\geq2$ and the result
is true for all nilpotent Leibniz algebras of dimension smaller
than $\dim(\g)$. We distinguish the following two cases.\\
{\it Case} $1$. $Z(\g,\n)\subseteq\g^2$. Choose a one-dimensional
ideal $\mathfrak{k}$ of $Z(\g,\n)$. Then, invoking Proposition
\ref{prop3.6}, and using the induction hypothesis, we
have\vspace{-.15cm}
\begin{alignat*}{1}
\dim(&HL_2(\mathfrak{g},\mathfrak{n}))\leq\dim(HL_2(\f{\g}{\mathfrak{k}},\f{\n}{\mathfrak{k}}))
+\dim(\g\curlywedge\mathfrak{k})\\
&\leq\dim(HL_2(\f{\g}{(\g^2\cap\n)+\mathfrak{k}},
\f{\n}{(\g^2\cap\n)+\mathfrak{k}}))+2(\dim(\g^2\cap\n)-1)
d(\f{\g/\mathfrak{k}}{Z(\g/\mathfrak{k},\n/\mathfrak{k})})+2\dim(\f{\g}{\g^2})\\
&\leq\dim(HL_2(\f{\g}{\g^2\cap\n},\f{\n}{\g^2\cap\n}))
+2\dim(\g^2\cap\n)d(\f{\g}{Z(\g,\n)})
+2(\dim(\f{\g}{\g^2})-d(\f{\g}{Z(\g,\n)})).
\end{alignat*}
Since $\dim(\g/\g^2)=d(\g/Z(\g,\n))$, the result follows in this
case.\\
{\it Case} $2$. $Z(\g,\n)\nsubseteq\g^2$. Suppose that
$\mathfrak{a}$ is a vector space complement of $Z(\g,\n)\cap\g^2$
in $Z(\g,\n)$. Then
$(\g,\n)=(\mathfrak{k}\oplus\mathfrak{a},(\mathfrak{k}\cap\mathfrak{n})\oplus\mathfrak{a})$,
for some ideal $\mathfrak{k}$ of $\g$. It is straightforward to
check that
$\mathfrak{g}^2\cap\mathfrak{n}=\mathfrak{k}^2\cap\mathfrak{n}$
and $\g/Z(\g,\n)=\mathfrak{k}/Z(\mathfrak{k},\mathfrak{k}\cap\n)$.
Putting
$\overline{\mathfrak{k}\cap\n}=(\mathfrak{k}\cap\n)/[\mathfrak{k},\mathfrak{k}\cap\n]$
and $\overline{\mathfrak{k}}=\mathfrak{k}/\mathfrak{k}^2$, we let
$B$ be the quotient of
$((\overline{\mathfrak{k}\cap\n})\ast\mathfrak{a})\oplus(\overline{\mathfrak{k}}\ast\mathfrak{a})$
by the ideal generated by the elements $(\overline{n}\ast
a,-\overline{n}\ast a)$ and
$(a\ast\overline{n},-a\ast\overline{n})$, in which
$n\in\mathfrak{k}\cap\n$ and $a\in\mathfrak{a}$. It follows from
Corollary 3.3 and the induction hypothesis that\vspace{-.15cm}
\begin{alignat*}{1}
\dim(HL_2(&\mathfrak{g},\mathfrak{n}))=\dim(HL_2(\mathfrak{k}\oplus\mathfrak{a},
(\mathfrak{k}\cap\mathfrak{n})\oplus\mathfrak{a}))=
\dim(HL_2(\mathfrak{k},\mathfrak{k}\cap\mathfrak{n}))+\dim(HL_2(\mathfrak{a}))+\dim(B)\\
&\leq\dim(HL_2(\f{\mathfrak{k}}{\mathfrak{k}^2\cap\n},\f{\mathfrak{k}\cap\n}
{\mathfrak{k}^2\cap\n}))+2\dim(\mathfrak{k}^2\cap\n)d(\f{\mathfrak{k}}{Z(\mathfrak{k}
,\mathfrak{k}\cap\n)})
+\dim(HL_2(\mathfrak{a}))+\dim(B)\\
&=\dim(HL_2(\f{\mathfrak{k}}{\mathfrak{k}^2\cap\n},
\f{\mathfrak{k}\cap\n}{\mathfrak{k}^2\cap\n}))+2\dim(\mathfrak{g}^2\cap\n)
d(\f{\mathfrak{g}}{Z(\mathfrak{g},\n)})
+\dim(HL_2(\mathfrak{a}))+\dim(B).
\end{alignat*}
But, repeated application of Corollary 3.3 deduces\vspace{-.15cm}
\begin{alignat*}{1}
\dim(HL_2(\f{\g}{\g^2\cap\n},\f{\n}{\g^2\cap\n}))&
=\dim(HL_2(\f{\mathfrak{k}}{\mathfrak{k}^2\cap\n}\oplus\mathfrak{a},
\f{\mathfrak{k}\cap\n}{\mathfrak{k}^2\cap\n}\oplus\mathfrak{a}))\\
&=\dim(HL_2(\f{\mathfrak{k}}{\mathfrak{k}^2\cap\n},\f{\mathfrak{k}\cap\n}{\mathfrak{k}^2\cap\n}))
+\dim(HL_2(\mathfrak{a}))+\dim(B),
\end{alignat*}
and thus the result holds, as required.
\end{proof}
Taking $\mathfrak{n}=\mathfrak{g}$ in Theorem \ref{theo3.5}, we
get the following
\begin{corollary}\label{coro3.8}
Let $\mathfrak{g}$ be a finite dimensional nilpotent Leibniz
algebra. Then\vspace{-.15cm}
\[\dim(HL_2(\mathfrak{g}))\leq(\dim(\mathfrak{g})-\dim(\mathfrak{g}^2))^2
+2\dim(\mathfrak{g}^2)d(\f{\g}{Z(\g)}).\vspace{-.15cm}\] In
particular, the equality holds if and only if $\g$ is abelian.
\end{corollary}
\section{The main results}
 Let $(\g,\n)$ be a pair of Leibniz algebras such that
$\dim(\n)=n$ and $\dim(\g/\n)=m$. In our recent paper
\cite{edalatzadeh3}, we obtained an upper bound $n(n+2m)$ for the
dimension of $HL_2(\g,\n)$. We also determined the structure of
all pairs of finite dimensional nilpotent Leibniz algebras with
$\dim(HL_2(\g,\n))=n(n+2m)-k$, for which $k=0,1,2$.
\begin{theorem}[\cite{edalatzadeh3}]\label{theo3.9}
With the above assumptions and notations, if $\g$ is nilpotent,
then

$(i)$ $\dim(HL_2(\mathfrak{g},\mathfrak{n}))=n(n+2m)$ if and only
if $\mathfrak{g}$ is abelian.

$(ii)$ $\dim(HL_2(\mathfrak{g},\mathfrak{n}))=n(n+2m)-1$ if and
only if
$(\g,\n)=(\mathfrak{e}\oplus\mathfrak{a}(q),\mathfrak{e}^2)$, for
some $q\geq0$.

$(iii)$ $\dim(HL_2(\mathfrak{g},\mathfrak{n}))=n(n+2m)-2$ if and
only if $\g=\mathfrak{m}\oplus\n$, where $\mathfrak{m}$ is a
nilpotent Leibniz algebra with $\dim(\mathfrak{m}^2)=1$ and $\n$
is a one-dimensional central ideal.
\end{theorem}
In the following theorem, we characterize all pairs of finite
dimensional nilpotent Leibniz algebras with
$\dim(HL_2(\g,\n))=n(n+2m)-3$.
\begin{theorem}\label{theo3.10}
With the above assumptions and notations,
$\dim(HL_2(\g,\n))=n(n+2m)-3$ if and only if one of the following
cases occurs.

$(a)$ $\g$ is nilpotent with $\dim(\g^2)=2$ and $\n$ is a
one-dimensional ideal of $Z(\g)\cap\g^2$.

$(b)$ $\g$ is nilpotent with $\dim(\g^2)=1$ and $\n$ is a
two-dimensional ideal of $Z(\g)$ containing $\g^2$.
\end{theorem}
To prove this, we use the following proposition.
\begin{proposition}\label{coro2.5}
    Let $\g=\e\oplus\mathfrak{a}(q)$ and $\n$ be an ideal of $\g$. Then

    $(i)$ If $\e^2\subseteq\n\subseteq Z(\g)$, then
    $\dim(HL_2(\g,\n))=2\dim(\g)\dim(\n)-(\dim(\n)+1)^2+2$.

    $(ii) $If $\e=J_1, J_2$ or $H_1$ and $\dim(\n)=2$, then \vspace{.25cm}\\
    \vspace{1cm}~\hspace{2.5cm}$\dim(HL_2(J_1\oplus\mathfrak{a}(q),\n))=
    \begin{cases}
    4q & {\rm if}~
    \n\subseteq\mathfrak{a}(q),\\
    4q+1 & {\rm if}~ J_1^2\subseteq\n\subseteq Z(\g),\\
    2q+1 & {\rm otherwise}.
    \end{cases}$\vspace{-.9cm}\\
    \vspace{-.1cm}~\hspace{2.5cm}$\dim(HL_2(J_2\oplus\mathfrak{a}(q),\n))=
    \begin{cases}
    4(q+1) & {\rm if}~
    \n\subseteq\mathfrak{a}(q),\\
    4q+5 & {\rm if}~ J_2^2\subseteq\n\subseteq Z(\g),\\
    2q+3 & {\rm otherwise}.
    \end{cases}$\vspace{.2cm}\\
    \vspace{-.1cm}~\hspace{2.4cm}$\dim(HL_2(H_1\oplus\mathfrak{a}(q),\n))=
    \begin{cases}
    4(q+1) & {\rm if}~
    \n\subseteq\mathfrak{a}(q),\\
    4q+5 & {\rm if}~ H_1^2\subseteq\n\subseteq Z(\g),\\
    2(q+2) & {\rm otherwise}.
    \end{cases}$\vspace{-.cm}\\
\end{proposition}
\begin{proof}
$(i)$ Evidently, $\n$ can be represented as the direct sum
$\e^2\oplus\mathfrak{i}$, where $\mathfrak{i}$ is an ideal of
$\mathfrak{a}(q)$. Hence, applying Corollary 3.3, there is an
isomorphism \vspace{-.2cm}
\[HL_2(\g,\n)\cong HL_2(\e,\e^2)\oplus HL_2(\mathfrak{a}(q),\mathfrak{i})
\oplus(\e^{ab}\ast\mathfrak{i})\oplus\ds\f{\e^2\ast\mathfrak{a}(q)}{\langle
x\ast y,y\ast x~|~x\in \e^2,
y\in\mathfrak{i}\rangle}.\vspace{-.1cm}\] Since
$\e^2\ast\mathfrak{a}(q)$ is abelian, the denominator of the above
factor is isomorphic to $\e^2\ast\mathfrak{i}$. Now, by
considering the split exact sequence of abelian Leibniz algebras
$\e^2\ast\mathfrak{i}\rightarrowtail\e^2\ast\mathfrak{a}(q)\twoheadrightarrow
\e^2\ast\mathfrak{a}(q)/\mathfrak{i}$, and using Corollary
\ref{coro2.4}$(i)$,$(ii)$, we get the required result.

$(ii)$ We only prove the case $\e=H_1$, the other two are verified
similarly. If $\n\subseteq H_1$ or $\mathfrak{a}(q)$, then
invoking Corollaries 3.3 and \ref{coro2.4}$(ii)$, respectively, we
have\vspace{-.2cm}
$$\dim(HL_2(\g,\n))=\dim(HL_2(H_1,\n))+\dim(HL_2(
\mathfrak{a}(q),0))+\dim(\f{\n}{[H_1,\n]}\ast\mathfrak{a}(q))=2(q+2),\vspace{-.2cm}$$
$$\hspace{-1cm}\dim(HL_2(\g,\n))=\dim(HL_2(H_1,0))+\dim(HL_2(\mathfrak{a}(q),\n))+
\dim(H_1^{ab}\ast\n)=4(q+1).\vspace{-.2cm}$$ Also, if
$H_1^2\subseteq\n\subseteq Z(\g)$, then part $(i)$ implies that
$\dim(HL_2(\g,\n))=4q+5$. Finally, we assume $\n\nsubseteq Z(\g)$
and $\n=\langle z,c_1x+c_2y+t\rangle$, where
$t\in\mathfrak{a}(q)$, $c_1,c_2\in\F$ and at least one of the
scalars, say $c_1$, is non-zero. If we set $\mathfrak{k}$ to be
the semidirect sum of $\mathfrak{a}(q)$ by $\langle x\rangle$, it
is readily checked that $\mathfrak{k}$ is an abelian complement of
$\n$ in $\g$ and so, by Propositions \ref{prop2.2}$(iii)$ and
\ref{prop2.3},
$\dim(HL_2(\g,\n))=\dim(HL_2(\g))-\dim(HL_2(\g/\n))=2(q+2).$
\end{proof}
Now, we equipped to prove the theorem.
\begin{proof}
Assume $\dim(HL_2(\g,\n))=n(n+2m)-3$. By Theorem
\ref{theo3.9}$(i)$, $\g$ is non-abelian and by the nilpotency of
$\g$, $\n\cap Z(\g)\neq0$. Pick non-zero elements $x\in\g^2$ and
$e_1\in\n\cap Z(\g)$, $e_2,\cdots,e_n\in\n$ and
$f_1,\cdots,f_m\in\g-\n$ such that the set $A=\{e_1,\cdots,
e_n,f_1,\cdots,f_m\}$ is a basis of $\g$ and $x\in A$. Using the
relations $(2a)$-$(2d)$, it is routine to check that
$B=\{e_i\curlywedge e_j, e_i\curlywedge f_l, f_l\curlywedge
e_i~|~1\leq i,j\leq n,
 1\leq l\leq m\}$
is a generating set for the Leibniz algebra $\g\curlywedge\n$. We
divide the rest of the proof into two steps.

{\it Step} $1$. We claim that $\n$ is central.\\
By way of contradiction, suppose that $\dim([\g,\n])\geq1$. Then
$\dim(\g\curlywedge\n)\geq n(n+2m)-2$. If $\dim(Z(\g)\cap\n)\geq2$
and assume that $e_i\in Z(\g)\cap\n$ for some $2\leq i\leq n$,
then using the formulas $(3b)$ and $(3d)$, one can easily see that
$e_1\curlywedge x=x\curlywedge e_1=x\curlywedge e_i=e_i\curlywedge
x=0_{\g\curlywedge\n}$, which means that
$\dim(\g\curlywedge\n)\leq n(n+2m)-3$, an impossibility. Hence,
$Z(\g)\cap\n$ and similarly, $\g^2$ must be of dimension $1$. It
follows, in particular, that $Z(\g)\cap\n=$span$\{e_1\}$ and
$\g^2=$span$\{x\}$. Since $0\neq [\g,\n]\subseteq \n\cap\g^2$ and
$\g$ is nilpotent, we conclude that  $\g^2\cap Z(\g)\cap\n\neq0$.
Then $\g^2=Z(\g)\cap\n=[\g,\n]=$span$\{e_1\}$ and $e_1\curlywedge
e_1=0_{\g\curlywedge\n}$. Also, according to Proposition
\ref{prop2.3}$(ii)$, $\g=\mathfrak{e}\oplus\mathfrak{a}(q)$, for
some $q\geq0$. First, assume that $\n$ is non-abelian. Then there
exist elements $a,b\in\n$ (which we can assume that $a,b\in A$)
such that $e_1=[a,b]$. It is readily verified that if $a=b$, then
$y\curlywedge e_1=y\curlywedge[a,a]=0_{\g\curlywedge\n}$ for all
$y\in\g$ and if $a\neq b$, then the sets $\{a\curlywedge e_1,
e_1\curlywedge a, e_1\curlywedge b\}$ and $\{b\curlywedge e_1,
e_1\curlywedge b, e_1\curlywedge a\}$ are linearly dependent. But
both of these modes are contradictory to the assumption that
$card(B)\geq n(n+2m)-2$. So, $\n$ must be abelian. In particular,
$\n\neq\g$. We claim that $\g$ is a Lie algebra. Suppose that
$e_1=[a,a]$ for some $a\in\g-\n$, note that we can assume $a\in
A$. Since $\dim(\n)>1$ we can employ $e_2$ and relation $(3a)$
implies $e_2\curlywedge[a,a]=0$. Also using $(3c)$ the set
$\{e_1\cw e_2,e_1\cw a,a\cw e_1\}$ is linearly dependent, which
gives the impossibility  $\dim(\g\cw \n)\leq n(n+2m)-2$.
 This proves the claim and then
$\g=H(k)\oplus\mathfrak{a}(q)$ for some $k\geq1$ and $q\geq0$,
where $H(k)$ is Heisenberg Lie algebra of dimension $2k+1$, see
\cite{Edalatzadeh}. Recalling that $[\g,\n]=$span$\{e_1\}$, we can
assume that $e_1=[a,b]$ for some $a\in\g$ and $b\in\n$. Then image
of the adjoint map $ad_a:\g\lo\g$ is of dimension $1$ and
$\dim(C_{\g}(a))=\dim(\ker ad_a)=\dim(\g)-1$. Analogously,
$\dim(C_{\g}(b))=\dim(\g)-1$ and consequently, $\dim(C_{\g}(a)\cap
C_{\g}(b))=\dim(\g)-2$. We now distinguish the following two
cases.

Case $1$. $C_{\g}(a)\cap C_{\g}(b)$ is abelian. Since
$\g=(C_{\g}(a)\cap C_{\g}(b))\dot{+}span\{a,b\}$, it is inferred
that $Z(\g)= C_{\g}(a) \cap C_{\g}(b)$ and in consequence,
$1+q=2k+q+1-2$. Then $\g=H(1)\oplus\mathfrak{a}(q)$ and
$\dim(\n+Z(\g))=\dim(\n)+\dim(Z(\g))-\dim(Z(\g)\cap\n)\leq\dim(\g)=q+3$,
or equivalently, $2\leq\dim(\n)\leq3$. If $\dim(\n)=2$, then the
non-centrality of $\n$ together with Proposition
\ref{coro2.5}$(ii)$ yields that $\dim(H_2(\g,\n))=2(q+2)$, which
contradicts our assumption. If $\dim(\n)=3$, then $\g=\n+Z(\g)$
and since $\n$ is abelian, we conclude that $\g$ is abelian, which
is again a contradiction.

Case $2$. $C_{\g}(a)\cap C_{\g}(b)$ is non-abelian. Then $e_1 =
[y,z]$ for some $y,z\in C_{\g}(a)\cap C_{\g}(b)$ and so,
$e_1\curlywedge b=b\curlywedge e_1=0_{\g\curlywedge\n}$. As we can
assume that $b\in A$, the last results contradicts the assumption
that $\dim(\g\curlywedge\n)\geq n(n+2m)-2$. We therefore conclude
that $\n$ is central

{\it Step $2$.} Completion of the proof.\\
Since $\n\subseteq Z(\g)$ , it follows that
$\dim(\g\curlywedge\n)=n(n+2m)-3$ and $n\curlywedge y=y\curlywedge
n=0_{\g\curlywedge\n}$ for all $n\in\n$, $y\in\g^2$. If
$\n\cap\g^2$ is trivial, one sees that $\dim(\n)=\dim(\g^2)=1$,
forcing $\g=\mathfrak{m}\oplus\n$, where $\mathfrak{m}$ is
nilpotent with $\dim(\mathfrak{m}^2)=1$ and $\n$ is a
one-dimensional central ideal of $\g$. But this contradicts
Theorem \ref{theo3.9}$(iii)$. Hence, assume that
$\n\cap\g^2\neq0$. If $\dim(\n)\geq3$ or $\dim(\g^2)\geq3$, then
we again observe that at least four elements of the generating set
$B$ can be zero. So, one of the following situations occurs.

$(i)$ $\dim(\n)=\dim(\g^2)=1$. Then by Proposition
\ref{prop2.3}$(ii)$, $\g=\e\oplus\mathfrak{a}(q)$ for some
$q\geq0$ and $\n=\e^2$, which gives a contradiction to Theorem
\ref{theo3.9}$(ii)$.

$(ii)$ $\dim(\n)=1$ and $\dim(\g^2)=2$. In this case, $\g$ is
nilpotent with the derived subalgebra of dimension $2$ and
$\n\subseteq Z(\g)\cap\g^2$.

$(iii)$ $\dim(\n)=2$ and $\dim(\g^2)=1$. In this case, $\g$ is
nilpotent with the derived subalgebra of dimension $1$ and
$\g^2\subseteq\n\subseteq Z(\g)$.

$(iii)$ $\dim(\n)=\dim(\g^2)=2$. In this case, as before, we
regard that at least four elements in the set $B$ are vanished,
which contradicts the fact that $\dim(\g\curlywedge\n)=n(n+2m)-3$.

The converse of theorem follows from Propositions
\ref{prop2.2}$(ii)$ and \ref{coro2.5}$(i)$.
\end{proof}

 Let $(\g,\n)$ be a pair of Leibniz algebras. A Leibniz
crossed module $\delta:\m\lo\g$ is called a {\it relative stem
cover} of $(\g,\n)$ if $\delta(\m)=\n$, $\ker\delta\cong
HL_2(\g,\n)$ and $\ker\delta\subseteq Z(\g,\m)\cap[\g,\m]$,
where\vspace{-.2cm}
\[Z(\g,\m)=\{ m\in\m~|~^xm=m^x=0, {\rm for~all}~ x\in\g\} ~~ {\rm and}~~
[\g,\m]=\langle~^xm,m^x~|~x\in \g,m\in\m\rangle.\vspace{-.2cm}\]
One observes that the relative stem cover $\delta:\m\lo\g$ of the
pair $(\g,\g)$ is the usual stem cover of $\g$, which was
introduced by Casas and Ladra \cite{Casas2}. In this case,
$x\in\g$ acts on $m\in\m$ by $~^xm =[\bar x,m]~,~m^x=[m,\bar x]$,
where $\bar x$ is any element in the pre-image of $x$ via
$\delta$. They also showed that Leibniz algebras have at least one
stem cover. In \cite{edalatzadeh3}, the authors generalized this
result to an arbitrary pair of Leibniz algebras.

In continuation of the section, we will construct a relative stem
cover for the direct sum of two pairs of Leibniz algebras in terms
of given relative stem covers of them, which is a generalization
of \cite[Corollary 5.6]{RS} for Leibniz algebras. To do this, we
need the following lemma.
\begin{lemma}\label{lem3.1}
Let $\delta_i:\mathfrak{m_i}\lo\mathfrak{g_i}$, $i=1,2$, be two
Leibniz crossed modules and $\mathfrak{a}$ be an ideal of
$(\overline{\mathfrak{m}}_1\ast\overline{\mathfrak{g}}_2)
\oplus(\overline{\mathfrak{g}}_1\ast\overline{\mathfrak{m}}_2)$
generated by
$(\overline{m}_1\ast\overline{\delta_2(m_2)},-\overline{\delta_1(m_1)}\ast\overline{m}_2)$
and
$(\overline{\delta_2(m_2)}\ast\overline{m}_1,-\overline{m}_2\ast\overline{\delta_1(m_1)})$,
for all $m_1\in\mathfrak{m}_1$, $m_2\in\mathfrak{m}_2$. Set
$B=((\overline{\mathfrak{m}}_1\ast\overline{\mathfrak{g}}_2)
\oplus(\overline{\mathfrak{g}}_1\ast\overline{\mathfrak{m}}_2))/\mathfrak{a}$.
Then

$(i)$ $\mathfrak{m}=\mathfrak{m}_1\dotplus\mathfrak{m}_2\dotplus
B$ is a Leibniz algebra endowed with the Leibniz multiplication
defined by\vspace{-.15cm}
\[[(m_1,m_2,z),(m_1',m_2',z')]=([m_1,m_1'],[m_2,m_2'],
(\overline{m}_1'\ast\overline{\delta_2(m_2)},-\overline{\delta_1(m_1)}
\ast\overline{m}_2')+\mathfrak{a})\vspace{-.15cm}\] for all
$m_i,m_i'\in\mathfrak{m}_i$, $i=1,2$, and $z,z'\in B$.

$(ii)$ There is an action of $\g:=\g_1\oplus\g_2$ on
$\mathfrak{m}$ defined by\vspace{-.15cm}
\[~^{(x_1,x_2)}(m_1,m_2,z)=(\hspace{-.1cm}~^{x_1}m_1,\hspace{-.1cm}~^{x_2}m_2,
(\overline{m}_x\ast\overline{x}_2,-\overline{x}_1
\ast\overline{m}_2)+\mathfrak{a}),\vspace{-.15cm}\]
\[\hspace{-.5cm}(m_1,m_2,z)^{(x_1,x_2)}=(m_1^{x_1}, m_2^{x_2},
(\overline{x}_2\ast\overline{m}_1,-\overline{m}_2
\ast\overline{x}_1)+\mathfrak{a})\vspace{-.15cm},\] for all
$m_i\in\mathfrak{m_i}$, $x_i\in\g_i$ $(i=1,2)$ and $z\in B$.

$(iii)$ The homomorphism $\delta:\mathfrak{m}\lo\g$,
$(m_1,m_2,z)\longmapsto(\delta_1(m_1),\delta_2(m_2))$, together
with the action given in Part $(ii)$ is a Leibniz crossed module.

$($Here
$\overline{\mathfrak{m}}_i=\mathfrak{m_i}/[\g_i,\mathfrak{m_i}],
\overline{\g_i}=\g_i/\g_i^2$, $i=1,2$, and the symbol $\dotplus$
denotes a direct sum of the underlying vector space structure.$)$
\end{lemma}
\begin{proof}
Parts $(i)$ and $(ii)$ follow straightforwardly. To prove the
final part of the lemma, we first show that
$\delta(^xm)=[x,\delta(m)]$ for all $m\in\mathfrak{m}$, $x\in\g$.
Suppose $m=(m_1,m_2,z)$ and $x=(x_1,x_2)$ for some
$m_i\in\mathfrak{m_i}$, $x_i\in\g_i$ $(i=1,2)$ and $z\in B$. Then
we have\vspace{-.15cm}
\begin{alignat*}{1}
\delta(^{(x_1,x_2)}(m_1,m_2,z))&=\delta(\hspace{-.1cm}~^{x_1}m_1,\hspace{-.1cm}~^{x_2}m_2,
(\overline{m}_1\ast\overline{x}_2,-\overline{x}_1\ast\overline{m}_2)+\mathfrak{a})=
(\delta_1(\hspace{-.1cm}~^{x_1}m_1),\delta_2(\hspace{-.1cm}~^{x_2}m_2))\\
&=([x_1,\delta_1(m_1)],[x_2,\delta_2(m_2)])=[(x_1,x_2),\delta(m_1,m_2,z)].
\end{alignat*}
Analogously, one can check the accuracy of other relations and the
proof is complete.
\end{proof}
\begin{theorem}\label{theo3.2}
Let $(\g_i,\n_i)$ be a pair of Leibniz algebras and
$\delta_i:\mathfrak{m_i}\lo\mathfrak{g_i}$ be a relative stem
cover of $(\g_i,\n_i)$ for $i=1,2$. Then the Leibniz crossed
module $\delta:\mathfrak{m}\lo\mathfrak{g}$ obtained in Lemma
$\ref{lem3.1}$ is a relative stem cover of the pair
$(\g_1\oplus\g_2,\n_1\oplus\n_2)$.
\end{theorem}
\begin{proof}
An easy verification shows that $\ker\delta=\ker\delta_1
\dotplus\ker\delta_2\dotplus B$ and Im$\delta=\n_1\oplus\n_2$. As
the actions of $\g_1$ on $\overline{\mathfrak{m}}_1$ and
$\overline{\g}_i$, $i=1,2$, are trivial, it follows that $\g_1$
acts trivially on $B$. Since a similar result holds for $\g_2$, we
deduce that $\ker\delta_1$, $\ker\delta_2$ and $B$ act trivially
on each other, whence
$\ker\delta=\ker\delta_1\oplus\ker\delta_2\oplus B$. Consequently,
it can be inferred from Corollary 3.3 that $\ker\delta\cong
HL_2(\g_1,\n_1)\oplus HL_2(\g_2,\n_2)\oplus B \cong
HL_2(\g_1\oplus\g_2,\n_1\oplus\n_2)$. It remains to prove that
$\ker\delta\subseteq Z(\g,\mathfrak{m}) \cap[\g,\mathfrak{m}]$.
Since $\ker\delta_i\subseteq Z(\g_i,\mathfrak{m_i})
\cap[\g_i,\mathfrak{m_i}]$ for $i=1,2$, the action of $\g$ on
$\ker\delta$ is trivial and, moreover, any element of $\ker\delta$
may be expressed as a finite linear combination of elements of the
form $(y_1,y_2,z)$, where $y_1=\hspace{-.1cm}~^{x_1}m_1$ or
$m_1^{x_1}$, $y_2=\hspace{-.1cm}~^{x_2}m_2$ or $m_2^{x_2}$, and
$z=(\bar m_1'\ast\bar x_2',\bar x_1'\ast\bar m_2')+\mathfrak{a}$,
$(\bar m_1'\ast\bar x_2',\bar m_2'\ast\bar x_1')+\mathfrak{a}$,
$(\bar x_2'\ast\bar m_1',\bar x_1'\ast\bar m_2')+\mathfrak{a}$ or
$(\bar x_2'\ast\bar m_1',\bar m_2'\ast\bar x_1')+\mathfrak{a}$. It
is seen that $(y_1,y_2,z)\in[\g,\mathfrak{m}]$ for all possible
cases of $y_1, y_2$ and $z$. For instance, we have\vspace{-.15cm}
\begin{alignat*}{1}
(\hspace{-.1cm}~^{x_1}m_1,m_2^{x_2},(\bar x_2'\ast\bar m_1',\bar
m_2'\ast\bar x_1')+\mathfrak{a})
&=~^{(x_1,0)}(m_1,0,0)+(0,m_2,0)^{(0,x_2)}+(m_1',0,0)^{(0,x_2')}\\
&+(0,m_2',0)^{(-x_1',0)}\in[\g,\mathfrak{m}].
\end{alignat*}
The proof is now complete.
\end{proof}
The following corollary is an immediate result of the above
theorem. This result is similar to the works of Wiegold
\cite{Wiegold}, and Salemkar and Edalatzadeh \cite{salemkar-ed} in
the cases of groups and of Lie algebras.
\begin{corollary}\label{coro3.3}
Let $\g_1$, $\g_2$ be any Leibniz algebras and $\mathfrak{m_1}$,
$\mathfrak{m_2}$ any covers of $\g_1$, $\g_2$, respectively. Then
$\mathfrak{m}=\mathfrak{m}_1\dotplus\mathfrak{m}_2\dotplus\mathfrak{m}_1^{ab}\ast\mathfrak{m}_2^{ab}$
with the Leibniz multiplication defined by\vspace{-.15cm}
\[[(m_1,m_2,z),(m_1',m_2',z')]=([m_1,m_1'],[m_2,m_2'],\bar m_1'
\ast\bar m_2-\bar m_1\ast\bar m_2')\vspace{-.15cm}\] for all
$m_i,m_i'\in\mathfrak{m}_i$, $i=1,2$, and
$z,z'\in\mathfrak{m}_1^{ab}\ast\mathfrak{m}_2^{ab}$, is a cover of
$\g_1\oplus\g_2$.
\end{corollary}
Theorem 4.5 may be useful to construct the relative stem covers of
some pairs of Leibniz algebras.
\begin{example}
$(i)$ With the assumptions of Theorem \ref{theo3.2}, if the
Leibniz algebras $\g_1$ and $\g_2$ are perfect, then
$\delta:\m_1\oplus\m_2\lo\g_1\oplus\g_2$ is a relative stem cover
of the pair $(\g_1\oplus\g_2,\n_1\oplus\n_2)$.

$(ii)$ Let $\g$ be a Leibniz algebra with ideals $\n$ and
$\mathfrak{u}$ such that $\g=\n\oplus\mathfrak{u}$ and let
$\delta_1:\m\lo\n$ be a relative stem cover of the pair $(\n,\n)$.
Since the map $\delta_2:0\lo\mathfrak{u}$ is a relative stem cover
of the pair $(\mathfrak{u},0)$, Theorem \ref{theo3.2} yields that
the map $\delta:\m\dotplus(\m^{ab}\ast\mathfrak{u}^{ab})\lo\g$,
$(m,z)\longmapsto\delta_1(m)$, is a relative stem cover of
$(\g,\n)$.

$(iii)$ Let $(\g,\n)$ be a pair of finite dimensional abelian
Leibniz algebras. It is easy to see that
$\delta_1:\n\dotplus(\n\curlywedge\n)\lo\n$, $(n,x)\longmapsto n$,
is a relative stem cover of the pair $(\n,\n)$, where $\n$ acts on
$\n\dotplus(\n\curlywedge\n)$ by $~^{n_1}(n,x)=(0,n_1\curlywedge
n)$ and $(n,x)^{n_1}=(0,n\curlywedge n_1)$, see \cite[Example
3.5]{edalatzadeh2}. So, by Part $(ii)$,\vspace{-.15cm}
\[\delta:\n\dotplus(\n\curlywedge\n)\dotplus(\n\ast\f{\g}{\n})\lo\g,~~~~~~
(n,x,y)\longmapsto n\vspace{-.15cm}\] is a relative stem cover of
$(\g,\n)$. But
$\dim((\n\curlywedge\n)\dotplus(\n\ast(\g/\n)))=(\dim(\n))^2+2
\dim(\n)\dim(\g/\n)$ is equal to $\dim(\g\curlywedge\n)$, thanks
to Corollary \ref{coro2.4}$(i)$. It therefore follows that
$(\n\curlywedge\n)\dotplus(\n\ast(\g/\n))\cong\g\curlywedge\n$ and
$\delta$ is a relative stem cover from
$\n\dotplus(\g\curlywedge\n)$ to $\g$ for the pair $(\g,\n)$.
\end{example}

\end{document}